\documentclass[12pt,leqno]{article}
\usepackage{graphicx}
\usepackage{amssymb, amscd, amsmath, amsthm}
\usepackage{dsfont}
\allowdisplaybreaks

\def\beq{\begin{equation}}
\def\eeq{\end{equation}}

\theoremstyle{definition}
\newtheorem{definition}{Definition}

\newtheorem*{remark}{Remark}

\theoremstyle{plain}
\newtheorem*{lemmaKl}{Kleitman Lemma}

\newtheorem{theorem}[definition]{Theorem}

\newtheorem{lemma}{Lemma}

\begin{document}

\title{New inequalities for families without $k$ pairwise disjoint members}

\author{Peter Frankl\footnote{R\'enyi Institute, Budapest, Hungary}, Andrey Kupavskii\footnote{Moscow Institute of Physics and Technology, University of Birmingham;  Email: {\tt kupavskii@yandex.ru}  \ \ Research supported by the grant RNF~16-11-10014 and by the EPSRC grant no. EP/N019504/1.}}


\date{}
\maketitle

\begin{abstract}
Some best possible inequalities are established for $k$\emph{-partition-free} families (cf. Definition~\ref{def:2}) and they are applied to prove a sharpening of a classical result of Kleitman concerning families without $k$ pairwise disjoint members.
\end{abstract}

\section{Introduction}
\label{sec:1}

Let $n$ be a positive integer, $[n] = \{1,2,\dots, n\}$ is the standard $n$-element set, $2^{[n]}$ its power set.
For an integer $k \geq 2$ a family $\mathcal F \subset 2^{[n]}$ is called $k${\it -dependent} if it contains no $k$ pairwise disjoint members.
Similarly, if $\mathcal F_1, \dots, \mathcal F_k \subset 2^{[n]}$ are not necessarily distinct families, we say that they are \emph{cross-dependent} if there is no choice of $F_i \in \mathcal F_i$, $i = 1,\dots, k$, such that $F_1, \dots, F_k$ are pairwise disjoint.

An important classical result of Kleitman \cite{Kl} determines the maximal size, $|\mathcal F|$ of a $k$-dependent family $\mathcal F \subset 2^{[n]}$ for the cases $n \equiv -1$ or $0$ (mod~$k$).
In a recent paper \cite{FK}, Kupavskii and the author determined the maximum of $|\mathcal F_1| + \ldots + |\mathcal F_k|$ for cross-dependent families $\mathcal F_i$ for all values of $n \geq k \geq 3$.
(Let us note that the easy case of $k = 2$ was already solved by Erd\H{o}s, Ko and Rado \cite{EKR}.)

\begin{definition}
\label{def:2}
For $k \geq 3$ and a family $\mathcal F \subset 2^{[n]}$ we say that $\mathcal F$ is $k$\emph{-partition-free} if $\mathcal F$ contains no $k$ pairwise disjoint members whose union is $[n]$.

Being $k$\emph{-partition-free} is slightly less restrictive than being $k$-dependent.
\end{definition}

For $0 \leq j \leq n$ let us use the notations $\mathcal F^{(j)} = \mathcal F \cap {[n]\choose j}$, $f^{(j)} = |\mathcal F^{(j)}|$.

The following inequality is an important discovery of Kleitman \cite{Kl}.

\begin{lemmaKl}
Let $\mathcal F \subset 2^{[n]}$ be $k$-partition-free and let $j_1, j_2, \dots, j_k$ be non-negative integers satisfying
$j_1 + \ldots + j_k = n$.
Then
\beq
\label{eq:1}
\sum_{1 \leq i \leq k}  \frac{f^{(j_i)}}{{n\choose j_i}} \leq k - 1.
\eeq
\end{lemmaKl}

The proof of \eqref{eq:1} is an easy averaging over all choices of pairwise disjoint sets $G_1, \dots, G_k$ satisfying $|G_i| = j_i$ and noting that at least one of the relations $G_i \in \mathcal F$ fails.

Since the relation $j_1 + \ldots + j_k = n$ is essential for proving \eqref{eq:1} it is rather surprising that in certain cases one can prove the analogous inequality even if $j_1 + \ldots + j_k > n$.

Let us first state our inequality for the case $k = 3$.

\begin{theorem}
\label{th:1}
Let $m > \ell > 0$ be integers, $n = 3m - \ell$.
Suppose that $\mathcal F \subset 2^{[n]}$ is $3$-partition-free.
Then
\beq
\label{eq:2}
\frac{|\mathcal F^{(m - \ell)}|}{{n\choose m - \ell}} + \frac{|\mathcal F^{(m)}|}{{n\choose m}} + \frac{|\mathcal F^{(m + \ell)}|}{{n\choose m + \ell}} \leq 2.
\eeq
\end{theorem}

Looking at the family ${[n]\choose m} \cup {[n]\choose m + \ell}$ shows that \eqref{eq:2} is best possible.

To state our most general result let us say that the families $\mathcal F_1,\dots, \mathcal F_k \subset 2^{[n]}$ are \emph{cross-partition-free} if there is no choice of $F_i \in \mathcal F_i$, $i = 1,\dots, k$ such that $F_1, \dots, F_k$ form a partition of~$[n]$.

\begin{theorem}
\label{th:3}
Let $m > \ell > 0$ be integers, $n = km - \ell$, $k \geq 3$.
For $1 \leq i \leq k$ let $\mathcal F_i \subset {[n]\choose m - \ell} \cup {[n]\choose m} \cup {[n]\choose m + \ell}$ and suppose that $\mathcal F_1, \dots, \mathcal F_k$ are cross-parti\-tion-free.
Then
\beq
\label{eq:3}
\sum_{1 \leq i \leq k} \frac{\bigl|\mathcal F_i^{(m - \ell)}\bigr|}{{n\choose m - \ell}} + \frac{\bigl|\mathcal F_i^{(m)}\bigr|}{{n\choose m}} + (k - 2) \frac{\bigl|\mathcal F_i^{(m + \ell)}\bigr|}{{n\choose m + \ell}} \leq (k - 1)k.
\eeq
\end{theorem}

Note that for $k = 3$ and $\mathcal F_1 = \ldots = \mathcal F_k$ the inequality \eqref{eq:3} implies \eqref{eq:2}.
The reason that we treat it separately is that both the statement and the proof are simple and hopefully give the reader the motivation to go through the more technical result~\eqref{eq:3}.

The proofs of \eqref{eq:2} and \eqref{eq:3} are based on Katona's cyclic permutation method (cf.\ \cite{Ka1}, \cite{Ka2}).

\numberwithin{lemma}{section}
\numberwithin{equation}{section}
\setcounter{equation}{0}

\section{The proof of \eqref{eq:2}}
\label{sec:1}

Let $x_1, x_2, \dots, x_{3m - \ell}, x_1$ be a random cyclic permutation of $\{1,2,\dots, n\}$ (as indicated above, the element after $x_n$ is $x_1$).
All $(n - 1)!$ cyclic permutations have the same probability $1/(n - 1)!$.
Set $d = (3m - \ell, m)$.

We define three families, $\mathcal B$, $\mathcal A$ and $\mathcal C$.
$\mathcal B = \bigl\{B_1, \dots, B_{(3m - \ell)/d}\bigr\}$  where $B_r = \{x_j : (r - 1) m < j \leq rm\}$,
$r = 1, \dots, (3m - \ell)/d$.
Note that the $B_r$ are arcs of $m$ consecutive elements~$x_j$.
Moreover, $(m, 3m - \ell) = d$ guarantees that each of the $(3m - \ell)/d$ arcs $B_r$ are distinct and the last element of $B_{n/d}$ is $x_n$.

Let us partition each $B_r$ as $B_r = A_r \cup D_r$ with $A_r$ being the arc consisting of the first $m - \ell$ elements.
Formally, $A_r = \{x_j : (r - 1) m < j \leq rm - \ell\}$, $D_r = B_r \setminus A_r$.
Set $C_r = B_r \cup D_{r + 1}$.
Define
\begin{align*}
\mathcal A &= \{A_r: \ 1 \leq r \leq n/d\},\\
\mathcal C &= \{C_r :\ 1 \leq r \leq n/d\}.
\end{align*}

Note that $C_r$ is not an arc but the union of two arcs and that it has the important property $C_r \cup A_{r + 1} = B_r \cup B_{r + 1}$ that we are going to use without further reference.

\begin{lemma}
\label{lem:2.1}
If $\mathcal F \subset 2^{[n]}$ is $3$-partition-free then
\beq
\label{eq:2.1}
|\mathcal F \cap \mathcal A| + |\mathcal F \cap \mathcal B| + |\mathcal F \cap \mathcal C| \leq 2n/d.
\eeq
\end{lemma}

\begin{proof}[Proof of \eqref{eq:2.1}]
Let $R = \{r : A_r \in \mathcal F\}$, $S = \{s : B_s \notin \mathcal F\}$, $T = \{t : C_t \notin \mathcal F\}$.
(We consider $S$ and $T$ as sets on distinct ground sets.)
To prove \eqref{eq:2.1} it is sufficient to show
\beq
\label{eq:2.2}
|R| \leq |S| + |T|.
\eeq
We prove \eqref{eq:2.2} by constructing an injection $\varphi$ from $R$ into $S \cup T$.

First note that $B_r, B_{r + 1}, A_{r + 2}$ form a partition of $[n]$.
This implies that if $r + 2 \in R$ then at least one of $r, r + 1$ is in $S$.
If $r + 1 \in S$, we set $\varphi(r + 2) = r + 1$.
If not then we let provisionally $\varphi(r + 2) = r$.

The only problem that might occur is that $r + 1$ is also in $R$ and therefore $\varphi(r + 2) = \varphi(r + 1) = r$.

Noting that $C_r, A_{r + 1}, A_{r + 2}$ form a partition of $[n]$, $r \in T$ follows.
We change the value of $\varphi(r + 2)$ to the element $r$ in~$T$.
This element is not allocated to any other $r' \in R$ and the proof of \eqref{eq:2.2} is complete.

To deduce \eqref{eq:2} from \eqref{eq:2.1} is easy averaging.
For every $B \in \mathcal B$ the probability of $B \in \mathcal F$ is $|\mathcal F^{(m)}|/{n\choose m}$ by the uniform random choice of the permutation.
The expected size $E(|\mathcal F \cap \mathcal B|)$ is
\[
|\mathcal B|\, |\mathcal F^{(m)}| \Bigm/{n\choose m} = \frac{|\mathcal F^{(m)}|}{{n\choose m}} \cdot \frac{n}{d}.
\]
The same holds for $\mathcal A$ and $\mathcal C$ as well.
By linearity of expectation and using the trivial fact that the expectation never exceeds the maximum, we infer
\[
\frac{n}{d} \left(\frac{|\mathcal F^{(m - \ell)}|}{{n\choose m - \ell}} + \frac{|\mathcal F^{(m)}|}{{n\choose m}} + \frac{|\mathcal F^{(m + \ell)}|}{{n\choose m + \ell}}\right) = E(|\mathcal F \cap \mathcal A| + |\mathcal F \cap \mathcal B| + |\mathcal F \cap \mathcal C|) \leq \frac{2n}{d}.
\]
Dividing by $\frac{n}{d}$ yields \eqref{eq:2}.
\end{proof}

\section{The proof of \eqref{eq:3}}
\label{sec:3}

The proof is similar to that of \eqref{eq:2} but both notationally and conceptually more complicated.
Set $d = (km - \ell, m)$ and $\overline n = n/d$.
Fix a random cyclic permutation $x_1, \dots, x_n$ of $\{1, \dots, n\}$ and define again the $\overline n$ arcs of length $m$,
$\mathcal B = \{B_1, \dots, B_{\overline n}\}$ where $B_r = \bigl\{x_q : (r - 1) m < q \leq rm\bigr\}$.
The choice of $\overline n$ guarantees that $B_{\overline n}$ ends with the element $x_n$.
This time we want to distribute these arcs among the $k$ families $\mathcal F_i$, $1 \leq i \leq k$.
For this reason let $b$ be the first positive integer such that $k$ divides $b\overline n$.
Of course, $b = k/(\overline n, k)$.

We let $B_r^{(p)}$ be a copy of $B_r$ and make a circle of $b\overline n$ sets in the following order:
$B_1^{(1)}, B_2^{(2)}, \dots, B_k^{(k)}, B_{k + 1}^{(1)}, \dots, B_{b\overline n}^{(k)}$.
For each pair $(r, p)$ we define the arc $A_r^{(p)} \subset B_r^{(p)}$ as the set of the first $m - \ell$ elements of $B_r^{(p)}$ and let $D_r^{(p)}$ be the rest: $D_r^{(p)} = B_r^{(p)} \setminus A_r^{(p)}$.
For $1 \leq j \leq k - 2$ we define the $(m + \ell)$-element sets
\[
C_r^{(p)}(j) = B_r^{(p)} \cup D_{r + j}^{(p + j)} \ \ \ \ (r + j \ \text{ is \ mod }\overline n, \ p + j \ \text{ is \ mod }k).
\]
Note that $C_r^{(p)}(j) \cup A_{r + j}^{(p + j)} = B_r^{(p)} \cup B_{r + j}^{(p + j)}$.
For $1 \leq p \leq k$ let us define $\mathcal B^{(p)} = \bigl\{B_1^{(p)}, \dots, B_{\overline n}^{(p)}\bigr\}$,
$\mathcal A^{(p)} = \bigl\{A_1^{(p)}, \dots, A_{\overline n}^{(p)}\bigr\}$ and
$\mathcal C_j^{(p)} = \bigl\{C_r^{(p)}(j): \ 1 \leq r \leq \overline n\bigr\}$, $1 \leq j \leq k - 2$.
Note that altogether we defined $\bigl(1 + 1 + (k - 2)\bigr)k = k^2$ families, each of size $b\overline n/k$.
Therefore \eqref{eq:3} will follow once we prove that out of these altogether $b\overline n k$ sets at most
$b\overline n(k - 1)$ are in the corresponding families~$\mathcal F_i$.
In other words we have to show that at least $b\overline n$ in total are missing.

Our plan is very simple.
Fixing an arbitrary pair $(i, r)$, $1 \leq i \leq k$, $1 \leq r \leq \overline n$, we want to show that there is an integer, $0 < t \leq k$ such that out of the following $tk$ sets at least $t$ are missing from the corresponding~$\mathcal F_{i'}$.

The list is $A_r^{(i)}, A_{r - 1}^{(i - 1)}, ..., A_{r - t + 1}^{(i - t + 1)}$; $B_{r - 1}^{(i - 1)}, ..., B_{r - t}^{(i - t)}$;
$C_{r - 1}^{(i - 1)}(j), ..., C_{r - t}^{(i - t)}(j)$, $1 \leq j \leq k - 2$.

To achieve this goal we prove a slightly stronger assertion.
Since we do not need them for this statement, we remove the upper indices and let $C_r(j)$ denote the set $C_r^{(i)}(j)$ and the same with $B_r^{(i)}$, $A_r^{(i)}$.

\begin{lemma}
\label{lem:3.1}
Let $r$ be fixed and consider the following $k$ groups of sets.
$\mathcal G_1 = \{A_r, B_{r - 1}\}$, $\mathcal G_2 = \{A_{r - 1}, B_{r - 2}, C_{r - 2}(1)\}$, \dots, $\mathcal G_{k - 1} = \{A_{r - k + 2}, B_{r - k + 1},$\break
$C_{r - k + 1}(1), \dots, C_{r - k + 1}(k - 2)\}$.
Suppose that we have families $\mathcal H_i$, $\mathcal H_i \subset \mathcal G_i$, $1 \leq i < k$ such that we cannot find $k$ members of $\mathcal H_1 \cup \ldots \cup \mathcal H_{k - 1}$ which partition $A_r \cup B_{r - 1} \cup \ldots \cup B_{r - k + 1}$.
Then there exists $t$, $1 \leq t \leq k$ satisfying
\beq
\label{eq:3.1}
\sum_{1 \leq s \leq t} \bigl|\mathcal G_s \setminus \mathcal H_s\bigr| \geq t.
\eeq
\end{lemma}

\begin{proof}
First consider $\mathcal H_1$.
If $\mathcal H_1 \subsetneqq \mathcal G_1$ then \eqref{eq:3.1} holds with $t = 1$.
If $\mathcal H_1 = \mathcal G_1$ then the two members $A_r$ and $B_{r - 1}$ partition $A_r \cup B_{r - 1}$.
Arguing indirectly, suppose that \eqref{eq:3.1} does not hold and let $1 \leq t < k$ be the smallest integer such that
$A_r \cup B_{r - 1} \cup \ldots \cup B_{r - t}$ cannot be partitioned using the sets in $\mathcal H_1 \cup \ldots \cup \mathcal H_t$.

By our assumptions $t$ exists and the above considerations show $t > 1$ and $A_r \in \mathcal H_1$.
The minimality of $t$ implies the existence of members $H_i \in \mathcal H_i$, $1 \leq i < t$ such that
\[
A_r \cup H_1 \cup \ldots \cup H_{t - 1} = A_r \cup B_{r - 1} \cup \ldots \cup B_{r - (t - 1)}
\]
is a partition with $H_i \in \mathcal H_i$.
To conclude the proof we will prove that \eqref{eq:3.1} holds for~$t$.

First note that adding $B_{r - t}$ would make a partition of $A_r \cup B_{r - 1} \cup \ldots \cup B_{r - t}$, implying $B_{r - t} \notin \mathcal H_t$.
To exhibit $t - 1$ further missing sets let us note the following important feature about the partition $H_1 \cup \ldots \cup H_{t - 1} = B_{r - 1} \cup \ldots \cup B_{r - (t - 1)}$:
whenever a set $C_{r - s}(j)$ occurs it must come together with $A_{r - s + j}$ and the union of these two sets is $B_{r - s} \cup B_{r - s + j}$.
Consequently, altering the order of the $H_u$, we can break up the partition as
\[
B_{r - 1} \cup \ldots \cup B_{r - (t - 1)} = B_{u_1} \cup \ldots \cup B_{u_\ell} \cup \bigl(B_{u_{\ell + 1}} \cup B_{u_{\ell + 2}}\bigr) \cup \ldots \cup \bigl(B_{u_{t - 2}} \cup B_{u_{t - 1}}\bigr).
\]
To prove \eqref{eq:3.1} we show the existence of \emph{distinct} sets in $\bigcup\limits_{1 \leq s \leq t} \mathcal G_s \setminus \mathcal H_s$, one for $B_{u_i}$ and two for $B_{u_i} \cup B_{u_{i + 1}}$.

For $B_{u_i}$ note $C_{r - t}(u_i - r + t) \cup A_{u_i} = B_{r - t} \cup B_{u_i}$.
Since $A_r \cup B_{r - 1} \cup \ldots \cup B_{r - t}$ \emph{cannot} be partitional by members of the $\mathcal H_s$,
either $C_{r - t}(u_i - r + t)$ or $A_{u_i}$ is missing from the $\mathcal H_s$.
For the case of $B_v \cup B_w$ (to simplify notation, $r - t < v < w \leq r - 1$) first note that one of the corresponding sets in $\mathcal H_1 \cup \ldots \cup \mathcal H_{t - 1}$
that partition $B_v \cup B_w$
 is $A_w$ (the other is $C_v(w - v)$).
Consider two partitions of $B_{r - t} \cup B_v \cup B_w$.
\begin{align*}
C_{r - t}(w - r + t) &\cup B_v \cup A_w \ \ \text{ and }\\
C_{r - t}(v - r + t) &\cup A_v \cup B_w.
\end{align*}
Since at least one set must be missing from both, we are done.
Noting that the exhibited candidates for missing sets are all distinct, the proof of \eqref{eq:3.1} is complete.\hfill$\square$

Equipped with \eqref{eq:3.1} it is not hard to prove Lemma~\ref{lem:3.1}.
Starting at an arbitrary $r$ we find, say, $t_1$ consecutive ``groups'' with at least a total of $t_1$ missing sets,
$1 \leq t_1 \leq k$.
Then starting at $r - t_1$ we find $t_2$ such groups, etc.
Going around the circle (of length $b\overline n$) the last position of the last group might not be $r + 1$.
However, since there are only $b\overline n$ members after making no more than $k$ full rounds we definitely have two sets of groups starting at the same element, say $r'$.
That is for the $t_w$ in between, say $t_a, t_{a + 1}, \dots, t_{a + q}$ one has
$t_a + t_{a + 1} + \ldots + t_{a + q} = c \cdot b\overline n$ with $c$ a positive integer.
For these positions we exhibited altogether at least $cb\overline n$ missing sets and each of them is counted at most $c$ times.
Therefore there are at least $b \overline n$ missing sets, proving Lemma~\ref{lem:3.1}.
\end{proof}

Since Lemma~\ref{lem:3.1} implies \eqref{eq:3} by the same averaging argument as Lemma~\ref{lem:2.1} implied \eqref{eq:2}, the proof of Theorem~\ref{th:3} is complete.

\numberwithin{definition}{section}
\section{Applications}
\label{sec:4}

\begin{definition}
\label{def:4.1}
For positive integers $n \geq k \geq 3$ let $p(n, k)$ denote the maximum of $|\mathcal F|$ over all $\mathcal F \subset 2^{[n]}$ that are $k$-partition-free.
\end{definition}

\begin{theorem}
\label{th:4.2}
\beq
\label{eq:4.1}
p(km - 1, k) = \sum_{j \geq m} {km - 1\choose j},
\eeq
moreover the only $k$-partition-free family achieving equality in \eqref{eq:4.1} is $\bigl\{G \subset [km - 1] : |G| \geq m\bigr\}$.
\end{theorem}

Let us note that Kleitman \cite{Kl} proved the same bound for the somewhat stronger restriction that the family is without $k$ pairwise disjoint sets.
Also, Kleitman did not prove the uniqueness of the optimal family.

\begin{proof}
Since the case $m = 1$ is trivial, we suppose $m \geq 2$.
Our main tool is Theorem~\ref{th:3} applied with $\ell = 1$, $\mathcal F_1 = \ldots = \mathcal F_k \overset{\text{\rm def}}{=} \mathcal F$.
For the $k$-partition-free family $\mathcal F \subset 2^{[n]}$, $n = km - 1$ we get from \eqref{eq:3}:
\[
\frac{|\mathcal F^{(m - 1)}|}{{n\choose m - 1}} + \frac{|\mathcal F^{(m)}|}{{n\choose m}} + (k - 2)
\frac{|\mathcal F^{(m + 1)}|}{{n\choose m + 1}} \leq k - 1.
\]
Setting $y(j) = {n\choose j} - |\mathcal F^{(j)}|$ we obtain
\beq
\label{eq:4.3}
\frac{y(m - 1)}{{n\choose m - 1}} + \frac{y(m)}{{n\choose m}} + (k - 2) \frac{y(m + 1)}{{n\choose m + 1}} \geq 1.
\eeq
Note ${km - 1\choose m} = (k - 1){km - 1\choose m - 1}$ and for further use
\beq
\label{eq:uj4.3}
{km - 1\choose m - j + 1} \Bigm/ {km - 1\choose m - j} = \frac{(k - 1)m + j - 1}{m - j + 1} > k - 1 \ \text{ for } \ j \geq 2.
\eeq
Using ${km - 1\choose m + i} \geq {km - 1\choose m}$ (valid for $i < m$) \eqref{eq:4.3} yields the following inequality.
\beq
\label{eq:4.4}
y(m - 1) + \frac{1}{k - 1} y(m) + \frac{k - 2}{k - 1} y(m + 1) \geq {km - 1\choose m - 1}.
\eeq
Let us apply \eqref{eq:1} with
\[
(j_1, \dots, j_k) = (m - \ell, m, m, \dots, m, m + \ell - 1) \ \text{ for } \ \ell = 2,3, \dots, m.
\]
Multiplying both sides by ${n\choose m - \ell}$ we obtain
\beq
\label{eq:4.5}
y(m - \ell) + \frac{(k - 2)}{(k - 1)^\ell} y(m) + \frac{1}{(k - 1)^\ell} y(m + \ell - 1) \geq {km - 1\choose m - \ell}
\eeq
$\left(\text{we used \eqref{eq:uj4.3} and } {n\choose m + \ell - 1} > {n\choose m}\right)$.
We want to add \eqref{eq:4.4} and the sum of \eqref{eq:4.5} over $2 \leq \ell \leq m$.
For $\ell > 2$ the term $y(m + \ell - 1)$ occurs only once and its coefficient is smaller than $\frac{1}{k - 1} < 1$.
The term $y(m + 1)$ has coefficient
\[
\frac{k - 2}{k - 1} + \frac{1}{(k - 1)^2} < \frac{k - 2}{k - 1} + \frac{1}{k - 1} = 1 \ \text{ also.}
\]
Finally $(k - 1)^{-2} + (k - 1)^{-3} + \ldots = \frac{k - 1}{k - 2} \cdot \frac{1}{(k - 1)^2} = \frac{1}{(k - 2)} \cdot \frac{1}{k - 1}$.
Thus the total coefficient of $y(m)$ will be less than $\frac{2}{k - 1} \leq 1$.
That is, we obtain an inequality of the form
\[
y(0) + y(1) + \ldots + y(m - 1) + c_m y(m) + \ldots + c_{2m} y(2m) \geq \sum_{0 \leq j < m} {n\choose j}
\]
with $c(m + i) < 1$ for $0 \leq i \leq m$.
Consequently, $|\mathcal F| \leq 2^n - \sum\limits_{0 \leq j < m} {n\choose j} = \sum\limits_{j \geq m} {n\choose m}$, as desired.
Moreover, in case of equality, $y(m + i) = 0$ must hold because of $c_{m + i} < 1$ for all $0 \leq i \leq m$.
Plugging these values into \eqref{eq:4.4} and \eqref{eq:4.5}, $y(m - \ell) = {n\choose m - \ell}$ follows for all $1 \leq \ell \leq m$.
That is, $\mathcal F = \{F \subset [n] : |F| \geq m\}$ concluding the proof of the uniqueness.
\end{proof}

\begin{remark}
If we used \eqref{eq:1} instead of Theorem~\ref{th:3} then instead of \eqref{eq:4.4} we would have
\[
y(m - 1) + y(m) \geq {km - 1\choose m}.
\]
Thus adding more equalities to it would make the coefficient of $y(m)$ greater than~$1$.
\end{remark}

\begin{definition}
\label{def:4.3}
The not necessarily distinct families $\mathcal F_1, \dots, \mathcal F_k$ are called \emph{cross-dependent} if there is no choice of $F_1 \in \mathcal F_1, \dots, F_k \in \mathcal F_k$ that are pairwise disjoint.
\end{definition}

Let us recall the following recent result of Kupavskii and the author.

\begin{theorem}[{\cite{FK}}]
\label{th:4.4}
Suppose that $\mathcal F_1, \dots, \mathcal F_k \subset 2^{[n]}$, $n = mk - 1$, are cross-dependent.
Then one has:
\beq
\label{eq:4.6}
|\mathcal F_1| + \ldots + |\mathcal F_k| \leq k \cdot \sum_{j \geq m} {mk - 1\choose j}.
\eeq
\end{theorem}

One can use Theorem~\ref{th:3} to prove \eqref{eq:4.6} under the weaker assumption of being cross-partition-free and show that equality holds only if
\[
\mathcal F_1 = \ldots = \mathcal F_k = \{F \subset [n] : |F| \geq m\}.
\]

We leave the details to the interested reader.

Let us mention that in \cite{FK} the maximum of $|\mathcal F_1| + \ldots + |\mathcal F_k|$ is determined for \emph{all} values of $n$ and~$k$.
The methods presented in this paper seem to be insufficient to tackle the cases $n \not \equiv -1 \, (\text{\rm mod }k)$.

\end{document}